\documentclass[12pt, reqno]{article} 
\usepackage{amsfonts}
\usepackage{amsthm}
\usepackage{amscd}
\usepackage{amssymb}
\usepackage{graphics}
\usepackage{amsmath}
\usepackage[english]{babel}
\usepackage{hyperref}
\usepackage{bbm}
\usepackage{yfonts}
\usepackage{mathrsfs}
\usepackage{ upgreek }
\usepackage{color}
\usepackage{epsfig}
\usepackage{graphicx}
\usepackage{subfig}

\usepackage{amsthm}
\usepackage{graphicx}
\usepackage{verbatim}
\usepackage{float}
\usepackage{fullpage}
\usepackage{amsfonts}
\usepackage{amscd}

\allowdisplaybreaks[4]

\DeclareFontFamily{OML}{rsfs}{\skewchar\font'177}
\DeclareFontShape{OML}{rsfs}{m}{n}{ <5> <6> rsfs5 <7> <8> <9>
rsfs7 <10> <10.95> <12> <14.4> <17.28> <20.74> <24.88> rsfs10 }{}
\DeclareMathAlphabet{\mathfs}{OML}{rsfs}{m}{n}
\newtheorem{theorem}{Theorem}

\newtheorem{definition}[theorem]{Definition}

\newtheorem{lemma}[theorem]{Lemma}

\newtheorem{prop}{Proposition}[section]

\newtheorem{thm}[prop]{Theorem}
\newtheorem{lem}[prop]{Lemma}

\newtheorem{cor}[prop]{Corollary}

\newtheorem{rem}[prop]{Remark}
\newtheorem{prob}[prop]{Problem}

\newtheorem{clm}[prop]{Claim}

\newcommand{\BE}{{\mathbb{E}}}

\newcommand{\BH}{{\mathbb{H}}}

\newcommand{\BN}{{\mathbb{N}}}

\newcommand{\BR}{{\mathbb{R}}}
\newcommand{\BS}{{\mathbb{S}}}

\newcommand{\BZ}{{\mathbb{Z}}}

\newcommand{\CE}{{\mathcal{E}}}
\newcommand{\CF}{{\mathcal{F}}}

\newcommand{\ind}{{\mathbbm{1}}}

\renewcommand{\prob}{{\bf P}}
\newcommand{\bae}{\begin{equation}\begin{aligned}}
\newcommand{\eae}{\end{aligned}\end{equation}}
\newcommand{\ev}{\mathbf{E}}
\newcommand{\pr}{\mathbb{P}}
\newcommand{\Z}{\mathbb{Z}}

\newcommand{\milf}{\text{SIDLA}}

\newcommand{\dd}{\partial}
\newcommand{\om}{{\omega}}

\newcommand{\lam}{{\lambda}}

\begin{document}

\numberwithin{equation}{section} \numberwithin{figure}{section}

\title{Stretched IDLA}

\author{Noam Berger\footnote{The Hebrew University of Jerusalem and Technische Universit\"at M\"unchen}, Jacob J. Kagan\footnote{The Weizmann Institute of Science}
, Eviatar B. Procaccia\footnotemark[\value{footnote}]}
\maketitle

\begin{abstract}
We consider a new IDLA - particle system model, on the upper half planar lattice, resulting in an infinite forest covering the half plane. We prove that almost surely all trees are finite. 
\end{abstract}

\maketitle



\section{Introduction}
The model of Internal Diffusion Limited Aggregation (IDLA) was introduced by Meakin and Deutch \cite{meakin1986formation} as a model for some chemical reactions, particle coalescence and aggregation. IDLA was first studied rigorously by Diaconis and Fulton \cite{diaconis1991growth} and by Lawler, Bramson and Griffeath \cite{1992}. IDLA is a growth model, starting with a point aggregate $0\in\BZ^2$, $A(0)=\{0\}$. At each step a particle exits the origin, performs a simple random walk (SRW) and stops at the first position outside the aggregate, this position is then added to the aggregate i.e. $A(n+1)=A(n)\cup v_n$, where $v_n$ is the first exit position of a SRW starting at $0$ from $ A(n)$. In \cite{1992}, Lawler, Bramson and Griffeath prove the asymptotic shape
of the IDLA aggregate converges to the Euclidean ball. Asselah and Gaudilli\`ere \cite{asselah2010note} and independently Jerison, Levine and Sheffield \cite{jerison2012logarithmic} recently proved the long standing conjecture, that the fluctuations from the Euclidean ball are at most logarithmic. 

In this paper we consider an IDLA process in continuous time, introduced to us by Itai Benjamini, which we call Stretched IDLA ($\milf$).  This process starts with an infinite line. Every vertex on the line has a Poisson clock, every ring initiates an oriented SRW that can add an edge to the tree rooted at the vertex whose clock rang. We show that even though eventually all vertices are covered, all trees are finite almost surely. See Figurs \ref{fig:finitetree} and \ref{fig:finitetree2} for two simulations of the process. The tree rooted at $0$ is colored red. In initiating the IDLA in an infinite line, we lose the simplicity of a discrete process, but we gain ergodicity which we use heavily in our analysis. Our main tool is coupling the $\milf$ to some first passage percolation model, with exponentially increasing weights, and proving all trees are finite in the percolation setting. 

A natural question that arises is universality of the finite tree property. In the last section we prove that all trees are finite in another first passage percolation model with exponentially decreasing weights. Another interesting problem is to characterize the decay of tree height. See Remark \ref{rem:zerner} for further discussion.

\begin{figure}
\centering
\subfloat[$\milf$ on $\Z^2$]{\includegraphics[width=0.45\textwidth]{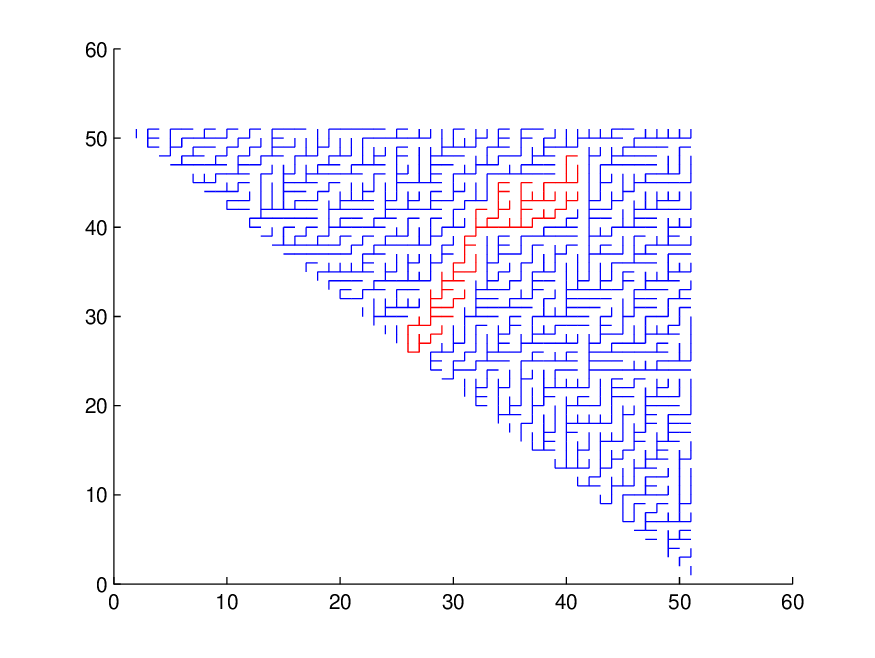}\label{fig:finitetree}}\quad
\subfloat[$\milf$ on the rotated lattice $\BH$]{\includegraphics[width=0.45\textwidth]{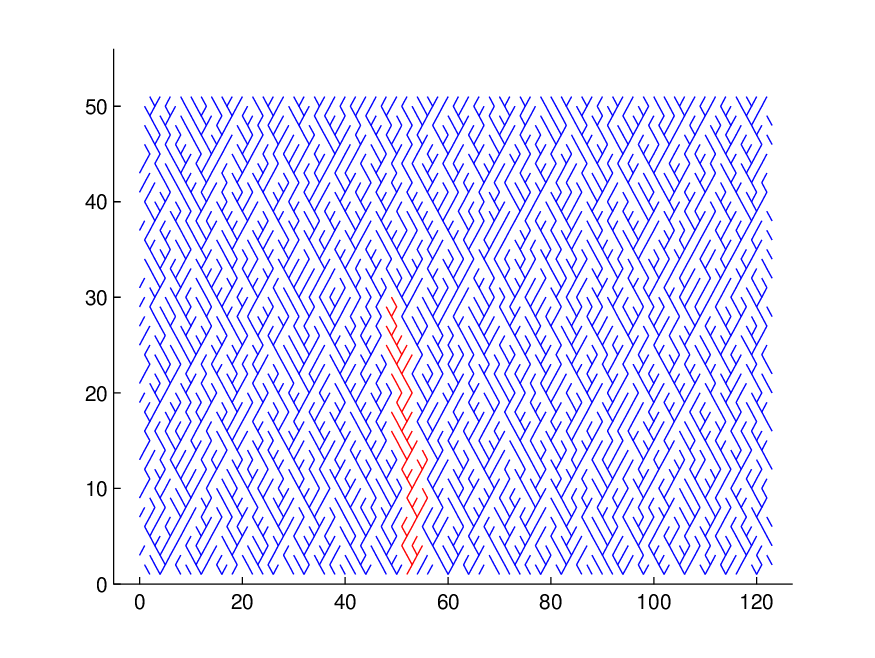}
\label{fig:finitetree2}}
\caption{Simulations of the SIDLA process.}
\end{figure}
       
\subsection{General Notation}
 We consider the rotated $\Z^2$ lattice in the upper half-plane re-scaled by $\sqrt 2$. Hereon we abbreviate it $\BH$,
$$\BH = \{(x,y)\in\BZ^2:x+y \in 2\cdot\Z,\:y\ge0  \} .$$
Denote by $\theta_l =(-1,1)$ and $\theta_r = (1,1)$ the vectors spanning the lattice.
Viewed as a directed graph, every site $v=(x,y)$ is connected to the sites $v+\theta_l=(x-1,y+1)$ and $v+\theta_r= (x+1,y+1)$. Abbreviate $\CE$ the set of edges in $\BH$. For a vertex $v = (x,y)$, denote the vertex height by $h(v) = y$. For an edge $e=(v,w)$, let $h(e)=\max\{h(v),h(w)\}$. 
 It will also be useful to define the cone of $v$, $C(v)=\{v+i\theta_l+j\theta_r:i,j\in\BN\cup\{0\}\}$, we write $e=(w,z)\in C(v)$ if $w,z\in C(v)$. This is the set of vertices and edges that can be reached from $x$ using directed edges. Finally we denote $\dd\BH = \{(x,0):x\in2\cdot\Z\}$.
See Figure \ref{fig:lattice} for a summary of the notation.

\begin{figure}
\begin{center}
\includegraphics[width=0.35\textwidth]{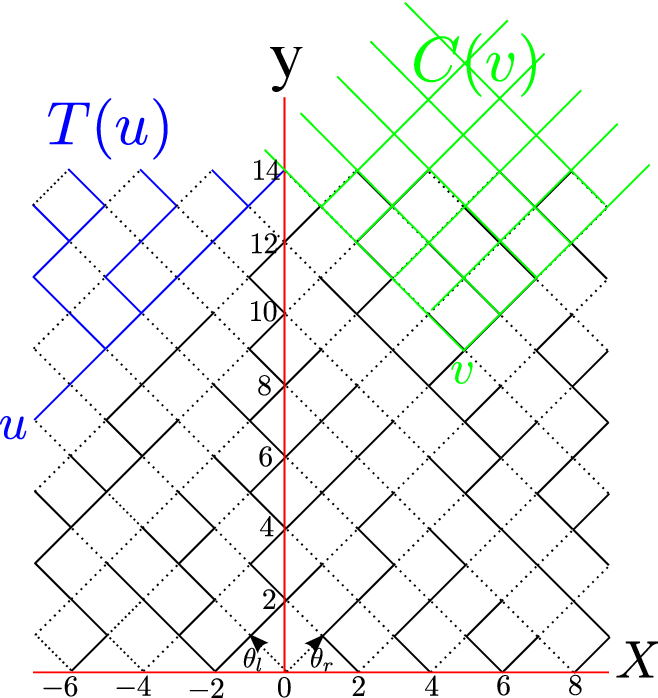}
\caption{The oriented lattice\label{fig:lattice}}
\end{center}
\end{figure}

A disjoint oriented rooted forest in $\BH$, is a collection of rooted trees $\{T(v)\}_{v\in\partial\BH}\subset\CE$, such that for every $v\neq v'$, $T(v)\cap T(v')=\emptyset$, and every rooted tree $T(v)$ is the union of oriented paths for the form $(e_1=(x_1,x_2),e_2=(x_2,x_3),\ldots,e_n=(x_n,x_{n+1}))$ starting from $x_1=v\in\partial\BH$ and $\forall i\le n,~x_{i+1}-x_i\in\{\theta_l,\theta_r\}$. For every tree $T(v)$ and a vertex $u\in\BH$, if there exists some $w\in\BH$ such that $(w,u)$ or $(u,w)$ is in $T(v)$, we abuse notations and say that $u\in T(v)$.

Let $T$ be an oriented tree in a disjoint oriented rooted forest. Denote by $\partial T$, the edge boundary of $T$ i.e. $\dd T = \{e = (u,v)\in\CE: v\notin T,  u\in T\}$. Abbreviate $\partial ^n T$ the boundary of height $n$, i.e. $\partial ^n T=\{e\in\partial T, h(e)=n\}$. The height of a tree is denoted $h(T)=\sup_{e\in\partial T}\{h(e)\}$. Denote by $T^n$ the vertices of height $n$ in $T$ i.e. $T^n = \{v| h(v)=n, v\in T\}.$ For any set $A\subset\BH$, let $\partial^{in}A=\{u\in A:\exists v\notin A, (u,v)\in\CE\text{ or }(v,u)\in\CE\}$.

We call a homogenous Poisson process $N(t)$ such that $N(t+\tau)-N(t)$ is distributed Poisson$(\lambda\tau)$ a Poisson clock of rate $\lambda$. When omitting the time of the Poisson clock we refer to the set of ring times i.e $N=\{t\in\BR^+:\forall s<t,~N(t)>N(s)\}$.
\subsection{SIDLA model description and general remarks}
In this section we give a description of the $\milf$, the well-definedness is proved below.

We construct the $\milf$ process on $\BH$. 
Let $\mathcal{F}$ be the set of disjoint oriented rooted forests in $\BH$, and let $\mathfs{F}$ be the $\sigma$-algebra spanned by the standard projection maps to $\CE$.  For every $t\ge 0$, let $\prob_t$ be a measure on $\CF$. 
The process starts with the empty forest i.e. $\prob_0(\forall v,T(v,0)=\emptyset)=1$.
Assume $\prob_t$ is defined and let $T(v,t)$ to be the tree rooted at $v$ sampled from $\prob_t$.
  

At each site $v$ found on the $x$ axis place an independent Poisson clock of rate 1. Given that a ring occurred at time $t_0>t$ an edge $e = (u_1,u_2)$ is adjoined to the tree according to the following law: 

\begin{equation*}
\prob_{t_0}(T(v,t_0)=T(v,t_0^-)\cup e) = \begin{cases}
 2^{-h(u_2)} & \text{if } u_1\in T(v,t_0^-)\text{ and } u_2\notin\bigcup_{\overset{v'\in\dd\BH}{v'\neq v}}T(v',t_0)\\
0 & \text{Otherwise}
\end{cases},
\end{equation*}
for every $e\neq e'$, $\prob_{t_0}(T(v,t_0)=T(v,t_0^-)\cup e\cup e')=0$,
where \[t_0^-=t_0^-(v)=\sup\{s>0:s<t_0, \text{clock at site }v\text{ rang at time }s\} .\]

This process can be described intuitively in terms of particles: each time $t_{0}$, the clock at a vertex $u\in\dd\BH$ rings, a particle is created, and starts an instantaneous oriented random walk subject to the following law: 

\begin{enumerate}
\item 
Being at vertex $v\in T(u,t_0^-)$, the particle chooses one of its neighbours $v+\theta_r$ and $v+\theta_l$ with probability $\frac{1}{2}$, call the choice $a$. 
\item
If $a$ is free, the particle occupies the edge $(v,a)$.
\item If $a\in \bigcup_{\overset{x\in\dd\BH}{x\neq u}} T(x,t_0)$ or $a\in T(u,t_0^-)$ but $(v,a)\notin T(u,t_0^-)$  the particles vanishes. 

\item 
Else it continues as described in (1.) from the newly reached vertex. 

\end{enumerate}
Since the process is defined in continuous time the question of well-definedness arises. However the geometry of $\BH$ greatly simplifies the matter. 
\begin{lem}
The process is well defined and $\prob_t$ converges strongly to a measure $\prob$ on disjoint oriented forests. 
\end{lem}
\begin{proof}
Each edge $e\in\CE$ can  a priori be reached only by a finite number of trees i.e. $|\{v\in\partial\BH:e\in C(v)|=h(e)$. For every $t>0$ we can order the rings of the Poisson clocks associated to the set of trees up to time $t$. For each ring we have an oriented random walk path, and $e$ can be joined to at most one tree. The well definedness of the process  for every $t\ge0$ follows. 

If some edge $e\in\CE$ is contained in some tree $T(v,t)$, then for every $s>t$, $e\in T(v,s)$ $\prob_s$-a.s. Thus the limit $\lim_{t\rightarrow\infty}\prob_t$ exists. Abbreviate the limiting measure $\prob$. \end{proof}

Let $T(v)=\lim_{t\rightarrow\infty}T(v,t)$. We can now state the main result of this paper: 
\begin{thm}\label{thm:main}
$\prob(\forall v\in\partial\BH, |T(v)|<\infty)=1$.
\end{thm}
\begin{rem} Note that every vertex in $\BH$ is reached at a finite time a.s. We use this remark in Corollary \ref{infforstmom} which states that the expected height of a tree in $\prob$ is infinity. 
\end{rem}

\subsection{First passage percolation}
In this section we define a first passage percolation model (FPP). In the next section we will couple the $\milf$ with the FPP defined in this section. 

Assign for each edge $e\in\CE$ a weight $\om(e)\sim\exp\left(2^{-h(e)}\right)$ independently of all other edges. We denote the measure on $[0,\infty]^\CE$ so constructed by $\pr$. For every oriented path $\gamma=(e_1,e_2,\ldots,e_n)$ in $\BH$, the length of $\gamma$ is defined to be $\lambda(\gamma)=\sum_{i=1}^n\om(e_i)$. For every two points $x,y\in\BH$ such that $x\in C(y)$ or $y\in C(x)$, let
\[
d_\om(x,y)=\min_{\gamma:x\rightarrow y}\lambda(\gamma)
,\] 
where the minimum is over all finite number of oriented paths in $\BH$ connecting $x$ and $y$. For a point $x\in\BH$ and a set $A\subset\BH$ connected by an oriented path, let $d_\om(x,A)=\inf_{y\in A}d_\om(x,y)$. 
\begin{definition}
For a vertex $x\in\partial\BH$, let $\hat{T}(x)=\bigcup_{y\in\BH}\{\gamma|\gamma \text{ is oriented, }\gamma:x\rightarrow y,\lam(\gamma)=d_\om(y,\partial \BH)\}$ i.e. the union of all oriented paths minimizing the distance from points $y\in\BH$ to $\partial\BH$ starting at the vertex $x$.
\end{definition}  
\begin{rem}
The uniqueness of the path $\gamma:x\rightarrow y,$ such that $\lam(\gamma)=d_\om(y,\partial\BH)$, follows from the independence and  continuity of the distribution of $\{\om(e)\}_{e\in\CE}$. 
\end{rem}
\begin{rem}
Since $\pr$ is a function of i.i.d. random variables, $\pr$ is ergodic under the shift $\theta:\BH \rightarrow\BH$ defined by $\theta(x)= x-\theta_l+\theta_r$.  
\end{rem}
\section{Coupling $\milf$ with FPP}\label{sec:coupling}

Given a FPP process with distribution $\pr$, we construct a $\milf$ process by way of coupling. The construction amounts to associating with each  $x\in\partial\BH$ a set of Poisson clock rings and prescribing the trajectory of each particle. 

To this end we introduce an auxiliary set of independent Poisson clocks. Given an edge $e\in\CE$ we associate with it a Poisson clock of rate $2^{-h(e)}$, which we abbreviate $\text{Poisson}(e)$, such that $\{\text{Poisson}(e)\}_e$ is an independent set of processes, and independent of the FPP measure $\pr$.

We assign a set of rings for $x$ and particle trajectories as follows: For each finite oriented path $\gamma\subseteq\hat{T}(x)\cup\dd \hat{T}(x)$, $\gamma=(e_1,\ldots,e_{l(\gamma)})$ originating at $x$ we assign the following rings: 

\begin{itemize}
\item 
if $\gamma \subset \hat{T}(x)$ we assign the ring $\sum_{i=1}^{l(\gamma)} \om(e_i)$, and the trajectory of the particle will be $\gamma$.

 \item
if $\gamma \nsubseteq \hat{T}(x)$ we assign the ring sequence $\sum_{i = 1}^{l(\gamma)} \om(e_i) $, $\sum_{i = 1}^{l(\gamma)} \om(e_i)+Poisson(e_{l(\gamma)})$, for each ring in this sequence of rings the particle will be assigned the path $\gamma$.
\end{itemize}

\begin{rem}
Note that in the second case, all the particles will vanish, as the vertex at the end of $\gamma$ will be reached sooner by a particle associated to the FPP tree containing it.
\end{rem}
 
We need to show that this construction results in a Poisson clock at $v$ for every $v \in \partial\BH$ with the correct rate. We prove this by showing that the time differences between every two consecutive rings is distributed exponentially with rate 1.


The next lemma is a combinatorial property of finite oriented trees in $\BH$. 
\begin{lem}\label{lem:combcrap}
For every finite oriented tree $T$ in $\BH$ with root $x\in\partial\BH$ and height $n-1$, then 
 \[
 \sum_{i=1}^n\frac{1}{2^i}|\partial^i T|=1
 .\]
\end{lem}
\begin{proof}
We prove by induction on $n$. For $n=1$, the tree is empty, thus $|\partial ^1 T|=2$ and for every $i>1$, $|\partial^i T|=0$. We get $\frac{1}{2}2=1$. Now assume the claim is true for $n-1$, let $T$ be a tree of height $n$. If $|\partial^1 T|=0$, denote by $T_r-\theta_r$ and $T_l-\theta_l$ the two subtrees of $T$ contained in $T\setminus\{x\}$ shifted to $\partial \BH$. The subtrees are of height smaller than $n$, and for every $i\le n$, $|\partial^i T_r|+|\partial^i T_l|=|\partial^{i+1} T|$ thus by the induction hypothesis    
\bae
\sum_{i=1}^n\frac{1}{2^i}|\partial^i T|=\sum_{i=1}^{n-1}\frac{1}{2^{i+1}}\left(|\partial^i T_r|+|\partial^i T_l|\right)=\frac{1}{2}+\frac{1}{2}=1.
\eae
If $|\partial^1 T|=1$, assume wlog $T_l=\emptyset$, by the induction hypothesis,
\bae
\sum_{i=1}^n\frac{1}{2^i}|\partial^i T|=\frac{1}{2}|\partial^1 T|+\sum_{i=2}^n\frac{1}{2^i}|\partial^i T|=\frac{1}{2}+\frac{1}{2}\sum_{i=1}^{n-1}\frac{1}{2^i}|\partial^i T_{r}|=1.
\eae 
\end{proof}
\begin{clm} The time differences between every two consecutive rings at any vertex v are independent and are distributed exponentially with rate $1$.
\end{clm}
\begin{proof}
By induction on the number of rings. The first ring happens at time $\min\{\om(e_r),\om(e_l)\}$ which are distributed exponentially $\om(e_r) \sim \exp(1/2)$, $\om(e_l)\sim \exp(1/2) $, thus their minimum, is distributed $\min\{\om(e_r),\om(e_l)\}\sim \exp(1)$.

Induction step: assuming the first $n$ rings have occurred, we  consider the $n+1^{st}$ interval of ring times. $T(v,t)$ after the $n$-th ring consists of at most $n$ vertices and edges, in particular $|T(v,t)|<\infty$. Let  $w'(e)$   be distributed according to $\pr$ independently from $\om$. By the memoryless property of exponential distribution, the $n+1^{\mbox{\tiny st}}$ interval between ring times is by definition of the coupling, distributed as $\min_{e\in\dd T(v,t)}w'(e)$. We prove by induction that
\bae
\mu_j=\min_{e\in\bigcup_{k=0}^{j}\partial^{n-k} T(v,t)}\{w'(e)\}\sim\exp\left(\frac{1}{2^{n-j}}\sum_{l=0}^{j}\frac{1}{2^{j-l}}\bigg|\partial ^{n-l}T(v,t)\bigg|\right)
.\eae
The base of induction follows as $\mu_0$  is the minimum of $|\partial ^{n}T(v,t)|$ , $\exp\left(\frac{1}{2^n}\right)$ independent random variables. Since 
\[\min_{e\in\partial^{n-j-1} T(v,t)}w'(e)\sim\exp\left(\frac{1}{2^{n-j-1}}\bigg|\partial ^{n-j-1}T(v,t)\bigg|\right),\]
\bae
\mu_{j+1}\sim\min\left\{\mu_j,\min_{e\in\partial^{n-j-1} T(v,t)}w'(e)\right\}\sim \exp\left(\frac{1}{2^{n-j-1}}\sum_{l=0}^{j+1}\frac{1}{2^{j+1-l}}\bigg|\partial ^{n-l}T(v,t)\bigg|\right)
.\eae
Thus proving the internal induction. We obtain by Lemma \ref{lem:combcrap} 
\bae
\mu_n\sim\exp\left(\sum_{l=0}^{n}\frac{1}{2^{n-l}}\bigg|\partial ^{n-l}T(v,t)\bigg|\right)\sim\exp(1).
\eae
\end{proof}


\section{Finite trees}
In this section we will prove the main result of this paper.
\begin{thm}\label{thm:finitefpp}
Given a FPP on $\BH$ distributed according to $\pr$, i.e. with weights $w(e)\sim\exp\left(2^{-h(e)}\right)$, almost surely all trees are finite, i.e.\[
\pr(|\hat{T}(0)|<\infty)=1
.\]  
\end{thm}
\begin{proof}
Assume for the purpose of contradiction the existence of an infinite tree. Then by shift invariance, $\beta:=\pr(|\hat{T}(0)|=\infty)>0$. 

Remember that $\hat{T}^m(x) = \{v| h(v)=m, v\in \hat{T}(x)\}.$ By the ergodic theorem we have
\bae
\frac{1}{2n+1}\sum_{x=-n}^{n}|\hat{T}^m(x)|\ind_{|\hat{T}(x)|=\infty}&\underset{n\rightarrow\infty}{\longrightarrow}\BE\left[|\hat{T}^m(0)|\big|\hat{T}(0)=\infty\right]\cdot\pr(|\hat{T}(0)|=\infty)\\
&=\beta\cdot\BE\left[|\hat{T}^m(0)|\big||\hat{T}(0)=\infty|\right].
\eae
Since all the trees are oriented, for every $x\in\partial\BH$, the tree $\hat{T}(x)$  resides in the cone $C(x)$. Thus
\[
\frac{1}{2n+1}\sum_{x=-n}^{n}|\hat{T}^m(x)|\ind_{|\hat{T}(x)|=\infty}\le\frac{1}{2n+1}\sum_{x=-n}^{n}|\hat{T}^m(x)|\le\frac{2n+2m+1}{2n+1}\underset{n\rightarrow\infty}{\longrightarrow}1,
\]
and we get 
\[
\BE\left[|\hat{T}^m(0)|\big|~|\hat{T}(0)|=\infty\right]\le\frac{1}{\beta}
.\]
Fix $\delta<1$,  $D=\frac{1}{\beta\cdot\delta}$, by Markov's inequality
\bae
\pr\left(|\hat{T}^n(0)|>D\big|\;|\hat{T}(0)|=\infty\right)\le\delta
\eae 
\begin{definition} A tree rooted at $v$ is called slim, if $0<|\hat{T}^n(v)|<D$ for infinitely many n's. We say that a tree is slim at level $k$ if $0<|\hat{T}^k(0)|<D$.
\end{definition}

$\hat{T}(0)$ is slim with probability greater than $1-\delta$ by the estimation 
\bae
\pr\left(\hat{T}(0)\text{ is not slim}\big|\;|\hat{T}(0)|=\infty\right)&= \pr\left(\bigcup_{n=1}^{\infty}\bigcap_{m=n}^{\infty}\{|\hat{T}^n(0)|>D\}\big|\;|\hat{T}(0)|=\infty\right)\\
&=\pr\left(\liminf_{n\rightarrow\infty}\{|\hat{T}^n(0)|>D\}\big|~ |\hat{T}(0)|=\infty\right)\\
&\le\liminf_n\pr\left(|\hat{T}^n(0)|>D\big|\;|\hat{T}(0)|=\infty\right)\le\delta.
\eae
By assuming existence of infinite trees we obtain a positive density of slim trees. We will reach a contradiction by showing that the probability of a tree being slim is $0$. 
\begin{definition}
Let $r_{n}=(\max\{s:(s,n)\in \hat{T}^n(0)\}+2,n)$ be the vertex to the right of $\hat{T}^{n}(0)$ and let $l_n$ be the vertex to the left of $\hat{T}^{n}(0)$. For every $n\in\BN$ denote $\Delta(n)=\BH\cap\text{Convex hull}\{l_{n},r_{n},l_{n}+\left(|\hat{T}^{n}(0)|+1\right)\theta_r\}$, the triangle based in $\hat{T}^{n}(0)\cup l_{n}\cup r_{n}$. See Figure \ref{fig:deltak} for clarifications. 
\end{definition}
\begin{lem}\label{lem:stocdomln}
For every $\kappa>1$, $\pr( d_\om(l_n,\partial\BH)>\kappa2^{n+1}|\sigma(\{\om(e):e\in\bigcup^n_{i=1} \hat{T}^i(0)\}))\le\frac{1}{\kappa}<1$ a.s.
\end{lem}
\begin{proof}
Let $w_i\sim\text{exp}(2^{-i})$, with law $Q$, be independent of each other and of $\pr$.
We first prove by induction on $n$ that $d_\om(l_n,\partial\BH)$ is stochastically dominated by $\sum_{i=1}^n w_i$. For $n=1$, if $T^1(0)$ is $\{\theta_r\}$, then $\om((\theta_l,0))>\om((l_1-\theta_r,l_1))$. $\om((l_1-\theta_r,l_1))$ is independent of $\hat{T}^1(0)$, and in particular $d(l_1,\partial\BH)$ is stochastically dominated by $w_1$. If $\hat{T}^1(0)$ is $\{\theta_l\}$ or $\{\theta_l,\theta_r\}$, $d(l_1,\partial\BH)=\min\{\om(-2,-2+\theta_l),\om(-4,-4+\theta_r)\}$, both are independent of $\hat{T}^1(0)$, and in particular dominated by $w_1$. Assume claim for $l_{n-1}$, if $l_n=l_{n-1}+\theta_l$, since there is no oriented path connecting $\hat{T}^n(0)$ with the edge $(l_{n-1},l_n)$, then $\om(l_{n-1},l_n)$ is independent of $\bigcup_{i=1}^n \hat{T}^i(0)$, and thus dominated by $w_n$. Since $d_\om(l_n,\partial\BH)\le d(l_{n-1},\partial\BH)+\om(l_{n-1},l_n)$, the claim follows by induction. If $l_n=l_{n-1}+\theta_r$, then $d_\om(l_n,\partial\BH)<d_\om(l_n-\theta_l,\partial\BH)+\om(l_n,l_n-\theta_l)$ . Thus conditioned on the weights of $\bigcup_{i=1}^{n-1} \hat{T}^i(0)$, and the structure of the tree, we obtain that
\bae\label{eq:stocdom}
0\le\om(l_{n-1},l_n)\le\om(l_n-\theta_l,l_n)+d(l_n-\theta_l,\partial\BH)-d(l_{n-1},\partial\BH) .\eae
Since the random variables on the RHS of \ref{eq:stocdom} are independent (without the conditioning) of $\om(l_{n-1},l_n)$, we obtain that $\om(l_{n-1},l_n)$ is conditionally dominated by $w_n$.
This is since for two independent random variables $X$ and $Y$, one has $\prob(X>t|X<Y)\le\prob(X>t)$.
Thus we get by the induction hypotheses that $d_\om(l_n,\partial\BH)\le d(l_{n-1},\partial\BH)+\om(l_{n-1},l_n)$ is stochastically dominated by $\sum_{i=1}^n w_i$. Now
\bae
\pr\left(d_\om(l_{n},\partial\BH)>\kappa2^{n+1}|\sigma(\{\om(e):e\in\bigcup^n_{i=1} \hat{T}^i(0)\})\right)
&\le{ Q\left(\sum_{i=1}^{n}w_i>\kappa E_Q\left[\sum_{i=1}^{n}w_i\right]\right)}\\
&\le\frac{1}{\kappa}<1. \eae 
\end{proof}

\begin{figure}[H]
\begin{center}
\includegraphics[width=0.4\textwidth]{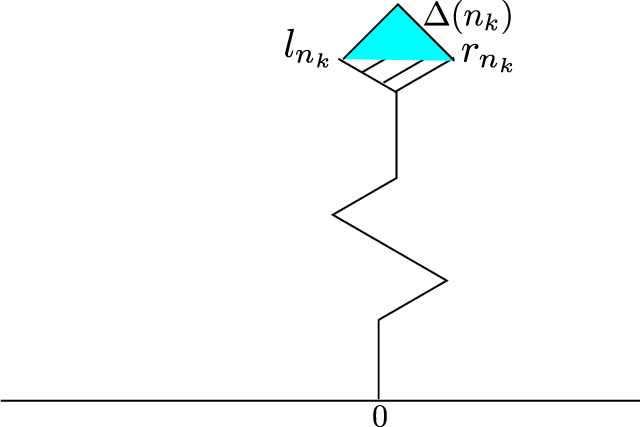}
\caption{Killing a slim tree.\label{fig:deltak}}
\end{center}
\end{figure}
 
Let $M_{n}=\max\{d_\om(l_{n},\partial\BH),d_\om(r_{n},\partial\BH)\}$.
Conditioned on the event that $\hat{T}(0)$ was slim in levels $n_1,\ldots,n_k$ such that $n_{m+1}-n_m>D$ and $2^{n_{k}+1}>M_{n_{k-1}}$, $m=1,\ldots,k-1$, we show that the probability there exists a level $l\ge n_k+D$ where the tree is slim is bounded away from 1.
 
Every edge $e\in\Delta(n_k)$ has weight distribution $\om(e)\sim\exp(2^{-h(e)})=\exp(2^{-n_k-l})$ where $0\leq l\leq D+1$. Using the exponential distribution properties $w(e)\sim2^{n_k}\exp(2^{-l})$. 

The idea that will follow is to show that with positive probability $\partial^{in} \Delta(n_k)\setminus \hat{T}^{n_k}(0)$ belongs to the union of the trees of $r_{n_k}$ and $l_{n_k}$ thus killing the tree rooted at $0$. To this end let $w_i\sim\text{exp}(2^{-i})$, be independent of each other and of $\pr$. We denote the measure so constructed by $Q$. By Lemma \ref{lem:stocdomln} (note that the conditioning is hiding in the notation $l_{n_k}$) we obtain that
\bae
\pr(d_\om(l_{n_{k}},\partial\BH)>M_{n_{k-1}}+\kappa2^{n_{k}+1})\le\pr(d_\om(l_{n_{k}},\partial\BH)>\kappa2^{n_{k}+1})\le\frac{1}{\kappa}<1. \eae 
With probability bounded away from zero and independent of all the levels lower than $n_{k}$, all (finite number) edges $e\in\Delta({n_k})$ will have weights larger than $\om(e)\ge\kappa2^{2D}\BE[\om(e)]\ge\kappa 2^{2D}2^{n_k}$, and all edges $e'=(x,y),\{x,y\}\in\partial^{in}\Delta({n_k})\setminus \hat{T}^{n_{k}}(0)$ will have weights smaller than $\om(e')\le\BE[\om(e')]$. Under this event, for every edge $e\in\partial^{in}\Delta({n_k})\setminus \hat{T}^{n_{k}}(0)$, $\omega(e)\le 2^{n_k+D}$. This yields,
\[
\sum_{e\in\partial^{in}\Delta({n_k})\setminus \hat{T}^{n_{k}}(0)} \omega(e)\le 2D\cdot2^{n_k+D}<\kappa 2^{2D}2^{n_k}
.\]
By the choice of $n_k$ we obtain that under the previous event $M_{n_k}+\sum_{e\in\partial^{in}\Delta({n_k})\setminus \hat{T}^{n_{k}}(0)} \omega(e)$ is smaller than the weight of a single edge in $\Delta({n_k})$, thus any geodesic that hits $\Delta({n_k})$ will not connect to $T^{n_k}(0)$.
We get that $\partial^{in}\Delta({n_k})\setminus \hat{T}^{n_{k}}(0)\notin \hat{T}(0)$. 
 
\end{proof}
\begin{cor}\label{infforstmom}
$\ev[h(\hat{T}(0))]=\infty$
\end{cor}
\begin{proof}
Assume for the purpose of contradiction that $\ev[h(\hat{T}(0))]<\infty$, thus 
\bae
\sum_{i=1}^\infty\prob(h(\hat{T}(0))\ge i)=\sum_{i=1}^\infty\prob(h(\hat{T}(i))\ge i)\le\frac{1}{2}\sum_{i=-\infty}^\infty\prob(h(\hat{T}(i))\ge |i|)<\infty.
\eae
By Borel-Cantelli, for all but a finite number $i$'s, $h(\hat{T}(i))<|i|$. Since all trees have finite height, there are infinitely many vertices in $C(0)$ that are not covered a.s. This is a contradiction to the construction of the $\milf$.   
\end{proof}
\begin{proof}[Proof of theorem \ref{thm:main}]
By the coupling of Section \ref{sec:coupling}, $\prob(|T(0)|<\infty)=\pr(|\hat{T}(0)|<\infty)$. By Theorem \ref{thm:finitefpp}, $\pr(|\hat{T}(0)|<\infty)=1$.
\end{proof}
\begin{rem}\label{rem:zerner}
An interesting question that so far evades rigorous proof is that of the correct decay of tree height. In \cite{MR1880239}, Zerner and Merkl presented a variation of the next forest model. Let $\mathbf{Z}$ be a measure on $\{0,1\}^\CE$ defined as follows: from each vertex $v\in\BH$ with $h(v)>0$, 
\bae
&\mathbf{Z}((v,v-\theta_r)=1,(v,v-\theta_l)=0)=\frac{1}{2}\\
&\mathbf{Z}((v,v-\theta_r)=0,(v,v-\theta_l)=1)=\frac{1}{2}
.\eae
Zerner and Merkl proved the that the height of trees under $\mathbf{Z}$ have a $\frac{1}{2}$ moment, by coupling an exploration process that surrounds the trees with two independent simple random walks. The tree is bounded by the trajectories of the random walks until the first time they meet.  
Since the SIDLA process is coupled to a FPP model with exponentially increasing weights, the law of the SIDLA is very close to $\mathbf{Z}$. We conjecture that SIDLA has $\frac{1}{2}-\epsilon$ moment for some small $\epsilon>0$. 
\end{rem}

\section{Analogous result for different FPP}
Once one sees the finite trees result for the FPP with exponentially increasing weights, one may ask if this phenomenon is preserved for different FPP measures e.g. a FPP with exponentially decreasing weights.

Let $\BS$ be a FPP measure on $\BH$ such that $\om(e)\sim\exp\left(2^{h(e)}\right)$, and abbreviate \[S(x)=\bigcup_{y\in\BH}\{\gamma,\text{ oriented }|\gamma:x\rightarrow y,l(\gamma)=d_\om(y,\partial \BH)\}.\] 
\begin{figure}[H]
\begin{center}
\includegraphics[width=0.5\textwidth]{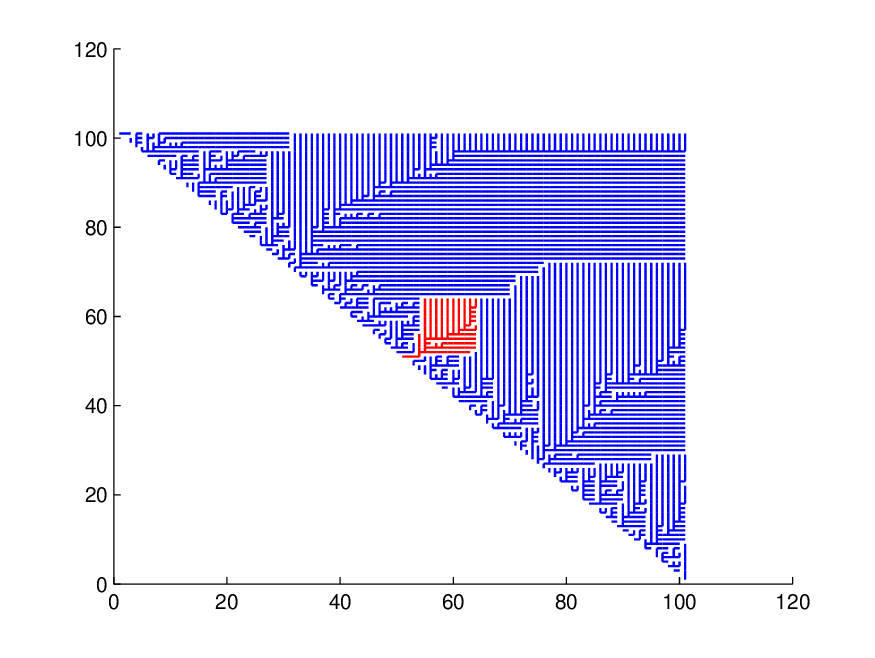}
\caption{FPP with decreasing weights\label{fig:idladecrease}}
\end{center}
\end{figure}
\begin{thm}
$S(0)$ is finite $\BS$ a.s.
\end{thm}
\begin{proof}
Denote by $a = \min\{\omega((0,\theta_r)),\omega((0,\theta_l))\}$. Let $l$ be the minimal integer such that 
\[\sum_{i=l}^{\infty}\left(2^{-i}+i\cdot2^{-i}\right)<\frac{a}{3}.\]
Consider $A^l(a)=\{0<v\in\dd\BH\: |\: \sum_{i=0}^{l-1} \omega\left((v+i\cdot \theta_l,v+(i+1)\cdot \theta_l)\right)<a/3 \}.$ Note that by shift invariance this set is infinite. 

\bae\label{eq:posray}
&\pr\left(\sum_{i=l}^{\infty} \omega\left((v+i\cdot \theta_l,v+(i+1)\cdot \theta_l)\right)<a/3\right)\ge
\\&\pr\left(\bigcap_{i=l}^{\infty}\left\{\omega((v+i\cdot\theta_l(v),v+(i+1)\cdot \theta_l))<2^{-i}+i\cdot2^{-i}\right\}\right)
\geq\ \prod_{i =l}^{\infty}\left(1-\frac{1}{i^2}\right)>0,\eae
where the one before last inequality follows from Chebyshev.

For every $v\in A^l(a)$, the events $\{\sum_{i=l}^{\infty} \omega\left((v+i\cdot \theta_l,v+(i+1)\cdot \theta_l)\right)<a/3\}$ and 
\[\left\{\sum_{i=0}^{l-1} \omega\left((v+i\cdot \theta_l,v+(i+1)\cdot \theta_l)\right)<a/3\right\},\] are independent. Thus by \eqref{eq:posray} There exists some $v\in A^l(a)$ such that 
\[
\sum_{i=0}^{\infty} \omega\left((v+i\cdot \theta_l,v+(i+1)\cdot \theta_l)\right)<\frac{2a}{3}<a,\]
thus the path $\bigcup_i\{v+i\cdot \theta_l\}\notin T(0)$. By symmetrical arguments there exists some $v'<0$ with $\bigcup_i\{v'+i\cdot \theta_r\}\notin T(0)$, thus $T(0)$ is finite. 
\end{proof}
\begin{rem}
An interesting open question is that of finite trees in the i.i.d case on $\BH$. i.e. $\om(e)\sim\exp\left(1\right)$. It has some relations to the Eden model on $\BH$. Similar to the Eden model each edge on the boundary of a tree is attempted to be added with equal probability. Under the coupling scheme of Section \ref{sec:coupling} bigger trees grow in a greater rate than smaller trees.  
\end{rem}
\section*{Acknowledgments}
The authors wish to thank Itai Benjamini for suggesting this problem and helpful discussions. One of the authors would like to thank Ohad Feldheim for a fruitful discussion.

\bibliography{IDLAT}

\begin{thebibliography}{LBG92}

\bibitem[AG10]{asselah2010note}
A.~Asselah and A.~Gaudilliere.
\newblock A note on fluctuations for internal diffusion limited aggregation.
\newblock {\em Arxiv preprint arXiv:1004.4665}, 2010.

\bibitem[DF91]{diaconis1991growth}
P.~Diaconis and W.~Fulton.
\newblock A growth model, a game, an algebra, lagrange inversion, and
  characteristic classes.
\newblock {\em Rend. Sem. Mat. Univ. Pol. Torino}, 49(1):95--119, 1991.

\bibitem[JLS12]{jerison2012logarithmic}
David Jerison, Lionel Levine, and Scott Sheffield.
\newblock Logarithmic fluctuations for internal {DLA}.
\newblock {\em J. Amer. Math. Soc.}, 25(1):271--301, 2012.

\bibitem[LBG92]{1992}
Gregory~F. Lawler, Maury Bramson, and David Griffeath.
\newblock Internal diffusion limited aggregation.
\newblock {\em The Annals of Probability}, 20(4):pp. 2117--2140, 1992.

\bibitem[MD86]{meakin1986formation}
P.~Meakin and JM~Deutch.
\newblock The formation of surfaces by diffusion limited annihilation.
\newblock {\em The Journal of chemical physics}, 85:2320, 1986.

\bibitem[ZM01]{MR1880239}
Martin P.~W. Zerner and Franz Merkl.
\newblock A zero-one law for planar random walks in random environment.
\newblock {\em The Annals of Probability}, 29(4):1716--1732, 2001.

\end{thebibliography}
\bibliographystyle{alpha}
\end{document}